\documentclass[aip,graphicx]{revtex4-1}

\draft 

\usepackage{dsfont} 
\usepackage{amssymb} 
\usepackage{mathrsfs} 
\usepackage{amsthm}
\usepackage{amsmath}
\usepackage{graphicx}
\usepackage{longtable}
\usepackage{graphicx}
\usepackage{subfigure}				
\usepackage{txfonts}									
\usepackage{amssymb}
\usepackage{changepage}
\usepackage{mathrsfs} 
\usepackage{color}

\begin{document}

\newtheorem{definition}{Definition}[section]
\newtheorem{lemma}{Lemma}[section]
\newtheorem{remark}{Remark}[section]
\newtheorem{theorem}{Theorem}[section]
\newtheorem{example}{Example}[section]
\newtheorem{corollary}{Corollary}[section]


\title{The role of slow manifolds in parameter estimation for a multiscale stochastic system with $\alpha$-stable L\'evy noise} 




\author{Ying Chao}
 \email{yingchao1993@hust.edu.cn}
 
\author{Pingyuan Wei}%
\altaffiliation[]{Corresponding author}
 \email{weipingyuan@hust.edu.cn}
\affiliation{School of Mathematics and Statistics \& Center for Mathematical Sciences \& Hubei Key Laboratory for Engineering Modeling and Scientific Computing, Huazhong
University of Science and Technology, Wuhan 430074,  China}

\author{Jinqiao Duan}
 \email{duan@iit.edu}
\affiliation{%
Department of Applied Mathematics, Illinois Institute of Technology, Chicago, IL 60616, United States of America}


\date{\today}

\begin{abstract}
This work is about parameter estimation for a fast-slow stochastic system with non-Gaussian $\alpha$-stable L\'evy noise. When the observations are only available for slow components, a system parameter is estimated and the accuracy for this estimation is quantified by $p$-moment with $p\in(1, \alpha)$, with the help of a reduced system through random slow manifold approximation. This method provides an advantage in computational complexity and cost, due to the dimension reduction in stochastic systems. To numerically illustrate this method, and to  corroborate that  the parameter estimator based on the reduced slow system is a good approximation for the true parameter value of the original system, a prototypical example is present.
\end{abstract}


\maketitle 

\section{Introduction}
\noindent
Multi-scale stochastic dynamical systems are ubiquitous in engineering and science. For example, slow and fast surface dynamics often occur in an electrocatalytic oscillator in an interactive way \cite{NNEH}; dynamics of  gene regulatory networks \cite{KJB, TGS, WTRY} usually evolve on notably different time scales, due to the fact that the production process of mRNA is faster than the protein dynamics.  More specifically,  the production of mRNA and proteins occur in an unpredictable and intermittent manner. These burst behaviors further contribute to variation or noise in individual cells or cell-to-cell interactions, which have been confirmed by  a large number of observations from biological experiments. Such perturbations appear to be appropriate modeled by the non-Gaussian noise. In fact, non-Gaussian random influences  are widely observed in many complex nonlinear systems \cite{BSW, Dror, YD, ZCZ}. Thus, it is significant and desirable to investigate two-scale stochastic differential equations (SDEs) under non-Gaussian (in particular, L\'evy type) fluctuations.\par

To make progress in understanding these complex dynamics, it is of a great importance to have a suitable tool for the reduction of such systems and their models to only their slow components, which is often essential for
scientific computation and further analysis. The reduction method based on the random slow manifold is one of such effective tool \cite{Cdz, FJP, FLD}.

We consider a type of reduction method for the multi-stochastic dynamical systems through the random slow manifolds in this paper. The theory of the random slow manifolds could serve as an effective tool for qualitative analysis of dynamical behaviors, 
as slow manifolds are geometric invariant structures in state space to examine or simplify stochastic dynamics \cite{FLD,SS}. For fast-slow stochastic dynamical systems in the context of Gaussian random fluctuations, the random slow manifolds have been utilized to investigate effective filtering on a reduced slow system \cite{QZD}, provide an accurate estimate on system parameter \cite{RDW}, detect the stochastic bifurcation \cite{HCDL} and understand certain chemical reactions \cite{HN}. The study of the dynamics generated by SDEs with non-Gaussian L\'evy noise is still in its infancy, but some interesting works are emerging \cite{ZCZ, CWY}. Under appropriate conditions, Yuan et al \cite{YD} have obtained low dimensional reduction of fast-slow stochastic dynamical systems driven by $\alpha$-stable L\'evy noise via random slow manifolds.\par

Parameter estimation issue is an important part of the overall multi-scale modeling strategy in a wide variety of applications. In general, when stochastic models are used to describe some certain phenomena, it is important to identify the unknown parameters in this model. For example, it's of great interest to examine the change rate for low-risk bounds in financial markets \cite{Dror}. And we are often interested in parameter (see \cite{WTRY}), which represents the degradation or production rates of protein and mRNA.  Recently, Zhang et al \cite{ZCZ} have devised a parameter estimator for a multi-scale SDEs with L\'evy noise by using stochastic averaging principle. In this paper, we develop a different parameter estimation method for multi-scale diffusions with non-Gaussian noise, with the help of the random slow manifolds.\par

In the present paper, we consider a fast-slow stochastic dynamical system under $\alpha$-stable L\'evy noise, but the observations are only available for slow components. By focusing on the reduced slow system on the random slow manifold, we demonstrate that an unknown system parameter in the drift term of such systems could be estimated. And the accuracy for this estimation is quantified by $p$-moment with $p\in(1, \alpha)$. Instead of solving original stochastic systems, this estimation method offers a benefit of dimension reduction in quantifying parameters in stochastic dynamical system. Furthermore, we verify this method and search for the estimated parameter value numerically, with the help of the stochastic Nelder-Mead method for optimization \cite{Chang}. Finally, we make some remarks in section 5.\par

This paper is organized as follows. After recalling some facts about random slow manifold and its approximation in the context of L\'evy random fluctuations in the next section, we devise a parameter estimator in section 3, by utilizing only observations on the slow component. Moreover, we establish the accuracy for this parameter estimator in terms of observation error and slow reduction error. And then in section 4, we illustrate our estimation method numerically in a specific example. Finally, we give some discussions and comments in a more biological context.\par

\renewcommand{\theequation}{\thesection.\arabic{equation}}
\setcounter{equation}{0}

\section{Preliminaries}
\noindent In this section, we recall some facts about L\'evy motions, introduce the framework for our reduction method for parameter estimation and present some results on slow manifold and its approximation.\par

\subsection{L\'evy process \cite{Duan, Sato}}

\begin{definition}
A stochastic process $L_t$ is a L\'evy process if\par
(1) $L_0=0$ (a.s.);\par
(2) $L_t$ has independent increments and stationary increments; and \par
(3) $L_t$ has stochastically continuous sample paths, i.e. for every $s\geqslant0$, $L_t\to L_s$ in probability, as $t\to s$.
\end{definition}
\noindent We now consider a special but important class of L\'evy motions, the $\alpha$-stable l\'evy motions, which are defined as follows.
\begin{definition}
For $\alpha\in(0,2)$, an n-dimensional symmetric $\alpha$-stable process $L_t^{\alpha}$ is a L\'evy process with characteristic function
$$\mathds{E}[\exp(i\left< u, L_t^{\alpha} \right>]=\exp\{-C_1(n, \alpha)|u|^{\alpha}\}, \text{ for } u\in \mathbb{R}^{n}$$
with $C_1(n, \alpha): = \pi^{-\frac{1}{2}}\Gamma((1+\alpha)/2)\Gamma(n/2)/\Gamma((n+\alpha)/2)$.
\end{definition}
\noindent Here, $\Gamma$ is the Gamma function. A useful fact is that for $\alpha\in(0,2)$, $\mathds{E}[|L_t^{\alpha}|^{p}]$ is finite according as $p\in(0, \alpha)$. We will quantify accuracy of estimation in terms of $p$-moment with $p\in(1, \alpha)$ in the next section.

\subsection{Framework \cite{Arnold,RDW, YD}}
\noindent
We consider the parameter estimation on $\lambda$ in the parameter space $\Lambda$ which is a closed interval of $\mathbb{R}$ in the following multi-scale stochastic dynamical system
\begin{eqnarray}
\dot{x}&=&\frac{1}{\varepsilon}Ax+\frac{1}{\varepsilon}f(x,y)+\sigma\varepsilon^{-\frac{1}{\alpha}}\dot{L_t^\alpha},\;\;\;\; x(0)=x_0 \in \mathbb{R}^n, \label{Equation-f-main}\\
\dot{y}&=&By+g(x,y,\lambda),\;\;\;\; y(0)=y_0 \in \mathbb{R}^m \label{Equation-s-main}.
\end{eqnarray}

\noindent The parameter $\varepsilon\ll 1$ represents the ratio of the two time scales. Here A and B are matrices, f, g are nonlinear Lipschitz continuous functions with Lipschitz constant $L_f$ and $L_g$ respectively, $\sigma$ is the intensity of noise and $L_t^\alpha$ is a two-sided $\mathbb{R}^n$-valued symmetric $\alpha$-stable L\'evy process defined on a probability space $(\Omega,\mathcal{F}, \mathbb{P})$ with index of the stability $1 < \alpha < 2$; refer to \cite{Duan, Sato, YD}. We remark that if $f$ and $g$ are only locally Lipschitz, but the corresponding deterministic system has a bounded absorbing set, we could obtain a modified system with globally Lipschitz drift by conducting a cut-off of the original system. Throughout the paper, we make the following hypotheses: \par
\noindent ({\bf H1}) There exists positive constants $\beta$, $\gamma$ and $K$, such that for every $x\in\mathbb{R}^n$ and $y\in\mathbb{R}^m$, the following exponential estimates hold:
$$|e^{At}x|_{\mathbb{R}^n}\leq Ke^{-\gamma t}|x|_{\mathbb{R}^n},\;t\geq 0;\;\;\;|e^{Bt}y|_{\mathbb{R}^m}\leq Ke^{\beta t}|y|_{\mathbb{R}^m},\;t\leq 0.$$
\noindent ({\bf H2}) $\gamma>KL_f$.\par

Before using the low dimensional reduction of system (\ref{Equation-f-main})-(\ref{Equation-s-main}) to estimate parameter $\lambda$, we give some results on random slow manifold and its approximation. We will treat the slow manifold under a driving flow $(\Omega, \mathcal{F}, \mathds{P}, \theta)$. 
\begin{definition}
 Let $(\Omega, \mathcal{F}, \mathds{P})$ be a probability space. And  $\theta$=$\{\theta_t\}_{t\in \mathbb{R}}$ is a flow on $\Omega$  which is defined as a mapping

 $$
 \theta: \mathbb{R}\times \Omega\mapsto \Omega
 $$
satisfying
\par
$\bullet$ $\theta_0=id  (indentity)~~on~~ \Omega; $
\par
$\bullet$ $\theta_{t_1}\theta_{t_2}=\theta_{t_1+t_2}~~~~for~~all~~ t_1, t_2 \in \mathbb{R}; $
\par
$\bullet$ the mapping $(t,\omega)\mapsto \theta_t\omega$ is $(\mathcal{B}(\mathbb{R})\otimes\mathcal{F},
\mathcal{F})-measurable$, where $\mathcal{B}(\mathbb{R})$ is the collection of Borels sets on the real line $\mathbb{R}$.
\end{definition}
\par

Now, introduce an auxiliary system 
$$\dot{z^\varepsilon}=\frac{1}{\varepsilon}Az^\varepsilon+\varepsilon^{-\frac{1}{\alpha}}\dot{L_t^\alpha}.$$
So by \cite[Lemma 3.1] {YD}, there exists a random variable $\eta^\varepsilon(\omega)=\varepsilon^{-\frac{1}{\alpha}}\int_{-\infty}^{0}e^{-\frac{As}{\varepsilon}}dL_s^\alpha(\omega)$ such that $\eta^\varepsilon(\theta_t\omega)=\varepsilon^{-\frac{1}{\alpha}}\int_{-\infty}^{t}e^{\frac{A(t-s)}{\varepsilon}}dL_s^\alpha(\omega)$
solves the above equation, where $\theta_t$ is the driving flow defined by $L_s^\alpha(\theta_t\omega)=L_{t+s}^\alpha(\omega)-L_t^\alpha(\omega)$. Set a random transformation
\begin{equation}\label{Transformation}
\binom{\hat{x}}{\hat{y}}=T(\omega,x,y):=\binom{x-\sigma\eta^\varepsilon(\omega)}{y},
\end{equation}
and then $(\hat{x}(t),\hat{y}(t))=T(\theta_t\omega,x(t),y(t))$ satisfy the following system with random coefficients
\begin{eqnarray}
\dot{\hat{x}}(t)&=&\frac{1}{\varepsilon}A\hat{x}+\frac{1}{\varepsilon}f(\hat{x}(t)+\sigma\eta^\varepsilon(\theta_t\omega),\hat{y}(t)), \label{Equation-f-rds}\\
\dot{\hat{y}}(t)&=&B\hat{y}(t)+g(\hat{x}(t)+\sigma\eta^\varepsilon(\theta_t\omega),\hat{y}(t),\lambda). \label{Equation-s-rds}
\end{eqnarray}

\noindent The following Lemma comes from \cite[Theorem 4.3] {YD}.\par
\begin{lemma}{\bf (Random slow manifold).}\label{lemma 2.1}
Assume that $\varepsilon>0$ is sufficiently small and ({\bf H1})-({\bf H2}) hold. Then the fast-slow system (\ref{Equation-f-main})-(\ref{Equation-s-main})
has a c\`{a}dl\`{a}g  random slow manifold $\mathcal{M}^\varepsilon(\omega)=\{(h^\varepsilon(\zeta, \omega),\zeta)| \; \zeta\in\mathbb{R}^m\}$. Here, $h^\varepsilon(\zeta, \omega): \mathbb{R}^m\rightarrow\mathbb{R}^n$ is a random nonlinear mapping expressed by $h^\varepsilon(\zeta, \omega)=\sigma\eta^\varepsilon(\omega)+\hat{h}^\varepsilon(\zeta, \omega)$, with $\hat{h}^\varepsilon(\zeta, \omega)$ determined as follows,
 $$\hat{h}^\varepsilon(\zeta, \omega)=\frac{1}{\varepsilon}\int_{-\infty}^{0}e^{-\frac{As}{\varepsilon}}f(\hat{x}(s,\omega,\zeta)+\sigma\eta^\varepsilon(\theta_s\omega), \hat{y}(s,\omega,\zeta))ds,\;\;\zeta\in\mathbb{R}^m.$$
\end{lemma}
\noindent In fact, the graph of the random mapping $\hat{h}^\varepsilon(\zeta, \omega)$ is the random slow manifold for the random system (\ref{Equation-f-rds})-(\ref{Equation-s-rds}).

Next, through the slow manifold $\mathcal{M}^\varepsilon(\omega)$ and by the same deduction as \cite[Corollary 4.1] {YD}, we could get a reduction system on $\mathcal{M}^\varepsilon(\omega)$. \par
\begin{lemma}{\bf (Reduced system on the random slow manifold).}\label{lemma 2.2}
Assume that $\varepsilon>0$ is sufficiently small and ({\bf H1})-({\bf H2}) hold. Then through the slow manifold $\mathcal{M}^\varepsilon(\omega)$, the system (\ref{Equation-f-main})-(\ref{Equation-s-main}) can be reduced to a lower dimensional slow stochastic system
 \begin{equation}\label{slowsystem}
\dot{y}=By+g(\hat{h}^\varepsilon(y, \theta_t\omega)+\sigma\xi(\omega),y,\lambda).
\end{equation}
where $\xi(\omega)=\int_{-\infty}^{0}e^{-As}dL_s^\alpha(\omega)$, through which $\xi(\theta_t\omega)=\int_{-\infty}^{t}e^{A(t-s)}dL_s^\alpha(\omega)$ is the stationary solution of linear system $\dot{z}(t)=Az+\dot{L_t^\alpha}.$
\end{lemma}
\noindent Here, $\xi(\omega)$ and $\eta^\varepsilon(\theta_{t}\omega)$ are identically distributed due to $\eta^\varepsilon(\theta_{\varepsilon t}\omega)$ and $\eta^\varepsilon(\omega)$ are identically distributed with $\xi(\theta_t\omega)$ and $\xi(\omega)$ respectively in Lemma 3.2 by \cite{YD}, together with the fact that $\eta^\varepsilon(\theta_{t}\omega)$
is identically distributed with $\eta^\varepsilon(\omega)$.\par

Finally, by time scaling for the system (\ref{Equation-f-rds})-(\ref{Equation-s-rds}) and a singular perturbation method, we could get a small $\varepsilon$ approximation for $\hat{h}^\varepsilon$; see \cite{YD,Cdz,Omalley,RDC}.
\begin{lemma}{\bf (An approximate slow manifold).}\label{lemma 2.3}
Assume the hypotheses of Lemma \ref{lemma 2.2} to be valid. Then there exists an approximate random slow manifolds in distribution for the system (\ref{Equation-f-rds})-(\ref{Equation-s-rds}). More precisely, $\hat{h}^\varepsilon(\zeta, \omega)=\hat{h}^{(0)}(\zeta, \omega)+\varepsilon\hat{h}^{(1)}(\zeta, \omega)+\mathcal{O}(\varepsilon^2)$, where
\begin{equation}\label{h0}
\hat{h}^{(0)}(\zeta, \omega)=\int_{-\infty}^{0}e^{-As}f(\hat{x}^{(0)}(s)+\sigma\xi(\theta_s\omega),\zeta)ds,
\end{equation}
and
\begin{align}\label{h1}
\hat{h}^{(1)}(\zeta, \omega)=&\int_{-\infty}^{0}e^{-As}\Big\{f_y\big(\hat{x}^{(0)}(s)+\sigma\xi(\theta_s\omega),\zeta\big)\Big[Bs\zeta+\int_{0}^{s}g\big(\hat{x}^{(0)}(r)+\sigma\xi(\theta_r\omega),\zeta,\lambda\big)dr\Big]\notag\\
&+f_x\big(\hat{x}^{(0)}(s)+\sigma\xi(\theta_s\omega),\zeta\big)\hat{x}^{(1)}(s)\Big\}ds.
\end{align}

\noindent Here, $\hat{x}^{(0)}(t)$ and $\hat{x}^{(1)}(t)$ solve the following random differential equations, respectively
 \begin{equation} \label{x0}
  \begin{split}
  &d\hat{x}^{(0)}(t)=A\hat{x}^{(0)}(t)dt+f(\hat{x}^{(0)}(t)+\sigma\xi(\theta_t\omega),\zeta)dt,   \\
  &\hat{x}^{(0)}(0)=\hat{h}^{(0)}(0,\omega),
  \end{split}
\end{equation}
and
 \begin{equation} \label{x1}
  \begin{split}
  d\hat{x}^{(1)}(t)=&\big[A+f_x(\hat{x}^{(0)}(t)+\sigma\xi(\theta_t\omega),\zeta)\big]\hat{x}^{(1)}(t)dt\\
  &+f_y\big(\hat{x}^{(0)}(t)+\sigma\xi(\theta_t\omega),\zeta\big)\Big[Bt\zeta+\int_{0}^{t}g(\hat{x}^{(0)}(s)+\sigma\xi(\theta_s\omega),\zeta,\lambda)ds\Big]dt,   \\
  \hat{x}^{(1)}(0)=&\hat{h}^{(1)}(0,\omega).
  \end{split}
\end{equation}
\end{lemma}

\noindent By combining Lemma \ref{lemma 2.2} and Lemma \ref{lemma 2.3}, thus we have an approximated slow system
\begin{equation}\label{aslowsystem}
\dot{y}=By+g\big(\tilde{h}^\varepsilon(y, \theta_t\omega)+\sigma\xi(\omega),y,\lambda\big),
\end{equation}
where $\tilde{h}^\varepsilon(y, \omega)=\hat{h}^{(0)}(y, \omega)+\varepsilon\hat{h}^{(1)}(y, \omega)$ is a first order approximation of $\hat{h}^\varepsilon(y, \omega)$.
This reduced system can be used to estimate parameter $\lambda$ in the next section.

\renewcommand{\theequation}{\thesection.\arabic{equation}}
\setcounter{equation}{0}

\section{Parameter estimation based on a random slow manifold}
 
\noindent In this section, we will estimate the unknown parameter $\lambda$ in the system (\ref{Equation-f-main})-(\ref{Equation-s-main}) based on the reduced slow system (\ref{slowsystem}) or (\ref{aslowsystem}). As we know, it is possible to make a good estimation of $\lambda$, when observations are available for both components $x$ and $y$ (see \cite{B, YangD}). In practice, it is often more feasible to observe slow variables than fast variables. So it's necessary to develop a parameter method by utilizing the observations on slow component $y$ only, which further reduces the computational complexity. \par

For the convenience of presentation, we introduce some notations. Denote the observation of the original slow-fast system (\ref{Equation-f-main})-(\ref{Equation-s-main}) with actual system parameter value $\lambda_0$ by $(x_{\lambda_0}^{ob}(t), y_{\lambda_0}^{ob}(t))$,  $t \in [0, T]$, and
the observation of the slow system (\ref{slowsystem}) with parameter $\lambda$ and initial value $y_0$ by $y_{\lambda}^S(t)$.
Define the objective function $F(\lambda)=\mathds{E}\int_{0}^{T}|y_{\lambda}^S(t)-y_{\lambda_0}^{ob}(t)|_{\mathbb{R}^m}^pdt$ with $p\in(1, \alpha)$ and assume that there is a unique minimizer $\lambda_E\in \Lambda$ such that $F(\lambda_E)=\min_{\lambda\in \Lambda}F(\lambda)$. In fact, this $\lambda_E$ is our parameter estimator. Besides, we further suppose that $g(x,y,\lambda)$ is Lipschitz continuous with respect to $x,y$ and $\lambda$ with Lipschitz constant $L_g$ and $\nabla_\lambda g(x,y,\lambda)$ is also a continuous function of $x,y$ and $\lambda$. The following theorem can provide an error estimation for this parameter estimation method.\par

\begin{theorem}\label{theorem 3.1} Set
$H(\lambda_0,\lambda_E):=\mathds{E}|\int_{0}^{t^\ast}e^{-Bt}\nabla_\lambda g(\hat{h}^\varepsilon(y_{\lambda'}^S(t), \theta_t\omega)+\sigma\xi(\omega),y_{\lambda'}^S(t),\lambda')dt|_{\mathbb{R}^m}$, where $\lambda'=\lambda_0+\kappa(\lambda_E-\lambda_0)$ with $\kappa\in(0,1)$ and choose $t^\ast\in [0, T]$ such that $\int_{0}^{T}|y_{\lambda_E}^S(t)-y_{\lambda_0}^{ob}(t)|_{\mathbb{R}^m}dt\geq T|y_{\lambda_E}^S(t^\ast)-y_{\lambda_0}^{ob}(t^\ast)|_{\mathbb{R}^m}$. Assume that $\varepsilon>0$ is sufficiently small. If $H(\lambda_0,\lambda_E)>0$ for $\lambda_E\in \Lambda$, then we obtain an error estimation for $\lambda_E$:
\begin{align}\label{equation-5}
|\lambda_0-\lambda_E|<&\frac{1}{H(\lambda_0,\bar{\lambda})}C^{\frac{1}{p}}(KT^{-\frac{1}{p}}+KL_gT^{\frac{p-1}{p}}+\frac{K^2L_gL_f}{\gamma-KL_f}T^{\frac{p-1}{p}})\notag\\
&\cdot\Big[\frac{C_1\varepsilon}{cp}\big(\mathds{E}\big|x_0-\sigma\xi(\omega)-\hat{h}^\varepsilon(y_0,\omega)\big|_{\mathbb{R}^n}^p\big)^{\frac{1}{p}}+F(\lambda_E)^{\frac{1}{p}}\Big],\notag
\end{align}
i.e. $|\lambda_0-\lambda_E|$ can be controlled by observation error $\mathcal{O}((F(\lambda_E))^{\frac{1}{p}})$ and the error due to slow reduction $\mathcal{O}(\varepsilon)$.
Here, $\bar{\lambda}\in \Lambda$ satisfies $H(\lambda_0,\bar{\lambda})=\max_{\lambda_E\in \Lambda}H(\lambda_0,\lambda_E)$ and $C$, $C_1$, $c$ are positive constants.\par
\end{theorem}

\begin{proof}
For $p\in (1,\alpha)$,
\begin{equation}
\mathds{E}|y_{\lambda_0}^S(t)-y_{\lambda_E}^{S}(t)|_{\mathbb{R}^m}^p\leq C(\mathds{E}|y_{\lambda_0}^{S}(t)-y_{\lambda_0}^{ob}(t)|_{\mathbb{R}^m}^p+\mathds{E}|y_{\lambda_E}^{S}(t)-y_{\lambda_0}^{ob}(t)|_{\mathbb{R}^m}^p)
\end{equation}
\noindent holds for some positive constant $C$. By integrating both sides with respect to time and Fubini's theorem, we get
\begin{equation}\label{pf2}
\int_{0}^{T}\mathds{E}|y_{\lambda_0}^S(t)-y_{\lambda_E}^{S}(t)|_{\mathbb{R}^m}^pdt\leq C(F(\lambda_0)+F(\lambda_E)),
\end{equation}
where $F(\lambda)=\mathds{E}\int_{0}^{T}|y_{\lambda}^S(t)-y_{\lambda_0}^{ob}(t)|_{\mathbb{R}^m}^pdt$ by definition.

\noindent According to \cite[Corollary 4.4] {YD}, there exist positive constants $C_1$ and $c$ such that for every $t\geq0$ and a.s. $\omega\in\Omega$,
$|y_{\lambda_0}^{S}(t)-y_{\lambda_0}^{ob}(t)|_{\mathbb{R}^m}\leq C_1e^{-\frac{ct}{\varepsilon}}|x_0-\sigma\xi(\omega)-\hat{h}^\varepsilon(y_0,\omega)|_{\mathbb{R}^n}$. Thus we have
\begin{equation}\label{pf3}
F(\lambda_0)\leq\frac{ C_1^p\varepsilon}{cp}\mathds{E}|x_0-\sigma\xi(\omega)-\hat{h}^\varepsilon(y_0,\omega)|_{\mathbb{R}^n}^p.
\end{equation}
\noindent Now we calculate the difference between $y_{\lambda_0}^{ob}(t)$ and $y_{\lambda_E}^S(t)$ to obtain

\begin{equation}
  \begin{split}
  \dot{y}_{\lambda_0}^S(t)-\dot{y}_{\lambda_E}^S(t)=&B(y_{\lambda_0}^S(t)-y_{\lambda_E}^S(t))+ \big[g(\hat{h}^\varepsilon(y_{\lambda_0}^S(t), \theta_t\omega)+\sigma\xi(\omega),y_{\lambda_0}^S(t),\lambda_0)\notag\\
  &-g(\hat{h}^\varepsilon(y_{\lambda_E}^S(t), \theta_t\omega)+\sigma\xi(\omega),y_{\lambda_E}^S(t),\lambda_E)\big],    \\
  y_{\lambda_0}^S(0)-y_{\lambda_E}^S(0)=&0.\notag
  \end{split}
\end{equation}

\noindent By the variation of constants formula, we have
\begin{align}
&e^{-Bt^\ast}[y_{\lambda_0}^S(t^\ast)-y_{\lambda_E}^S(t^\ast)]\notag\\
=&\int_{0}^{t^\ast}e^{-Bt}\big[g(\hat{h}^\varepsilon(y_{\lambda_0}^S(t), \theta_t\omega)+\sigma\xi(\omega),y_{\lambda_0}^S(t),\lambda_0)-g(\hat{h}^\varepsilon(y_{\lambda_E}^S(t), \theta_t\omega)+\sigma\xi(\omega),y_{\lambda_E}^S(t),\lambda_0)\big]dt
\notag\\
&+\int_{0}^{t^\ast}e^{-Bt}\big[g(\hat{h}^\varepsilon(y_{\lambda_E}^S(t), \theta_t\omega)+\sigma\xi(\omega),y_{\lambda_E}^S(t),\lambda_0)-g(\hat{h}^\varepsilon(y_{\lambda_E}^S(t), \theta_t\omega)+\sigma\xi(\omega),y_{\lambda_E}^S(t),\lambda_E)\big]dt,\notag
\end{align}
and then
\begin{align}\label{equation-1}
&\int_{0}^{t^\ast}e^{-Bt}[g(\hat{h}^\varepsilon(y_{\lambda_E}^S(t), \theta_t\omega)+\sigma\xi(\omega),y_{\lambda_E}^S(t),\lambda_0)-g(\hat{h}^\varepsilon(y_{\lambda_E}^S(t), \theta_t\omega)+\sigma\xi(\omega),y_{\lambda_E}^S(t),\lambda_E)]dt\notag\\
=&-\int_{0}^{t^\ast}e^{-Bt}[g(\hat{h}^\varepsilon(y_{\lambda_0}^S(t), \theta_t\omega)+\sigma\xi(\omega),y_{\lambda_0}^S(t),\lambda_0)-g(\hat{h}^\varepsilon(y_{\lambda_E}^S(t), \theta_t\omega)+\sigma\xi(\omega),y_{\lambda_E}^S(t),\lambda_0)]dt\notag\\
&+e^{-Bt^\ast}[y_{\lambda_0}^S(t^\ast)-y_{\lambda_E}^S(t^\ast)].\notag\\
\end{align}

Furthermore, via taking norm on two sides of (\ref{equation-1}) and using the mean value theorem, it holds that
\begin{align}\label{equation-2}
&|\lambda_0-\lambda_E|\cdot\Big|\int_{0}^{t^\ast}e^{-Bt}\nabla_\lambda g(\hat{h}^\varepsilon(y_{\lambda'}^S(t), \theta_t\omega)+\sigma\xi(\omega),y_{\lambda'}^S(t),\lambda')dt\Big|_{\mathbb{R}^m}\notag\\
\leq & \int_{0}^{t^\ast}Ke^{-\beta t}L_g\big[|\hat{h}^\varepsilon(y_{\lambda_0}^S(t), \theta_t\omega)-\hat{h}^\varepsilon(y_{\lambda_E}^S(t), \theta_t\omega)|_{\mathbb{R}^n}+|y_{\lambda_0}^{S}(t)-y_{\lambda_E}^S(t)|_{\mathbb{R}^m}\big]dt\notag\\
&+Ke^{-\beta t^\ast}|y_{\lambda_0}^{S}(t^\ast)-y_{\lambda_E}^S(t^\ast)|_{\mathbb{R}^m}\notag\\
<& KL_g\int_{0}^{T}|\hat{h}^\varepsilon(y_{\lambda_0}^S(t), \theta_t\omega)-\hat{h}^\varepsilon(y_{\lambda_E}^S(t), \theta_t\omega)|_{\mathbb{R}^n}dt+KL_g(T^{p-1}\int_{0}^{T}|y_{\lambda_0}^{S}(t)-y_{\lambda_E}^S(t)|_{\mathbb{R}^m}^pdt)^{\frac{1}{p}}\notag\\
&+K|y_{\lambda_0}^{S}(t^\ast)-y_{\lambda_E}^S(t^\ast)|_{\mathbb{R}^m},
\end{align}

\noindent where $\lambda'=\lambda_0+\kappa(\lambda_E-\lambda_0)$ with $\kappa\in(0,1)$ and the last step is based on exponential estimation property, Lipschitz continuity of $g$, together with H\"{o}lder inequality.\par
Note that $t^\ast$ satisfies $\int_{0}^{T}|y_{\lambda_0}^{S}(t)-y_{\lambda_E}^S(t)|_{\mathbb{R}^m}dt\geq T|y_{\lambda_0}^{S}(t^\ast)-y_{\lambda_E}^S(t^\ast)|_{\mathbb{R}^m}$,
thus by H\"{o}lder inequality

\begin{equation}\label{inequality1}
|y_{\lambda_0}^{S}(t^\ast)-y_{\lambda_E}^S(t^\ast)|_{\mathbb{R}^m}\leq\frac{1}{T}\int_{0}^{T}|y_{\lambda_0}^{S}(t)-y_{\lambda_E}^S(t)|_{\mathbb{R}^m}dt\leq (\frac{1}{T}\int_{0}^{T}|y_{\lambda_0}^{S}(t)-y_{\lambda_E}^S(t)|_{\mathbb{R}^m}^pdt)^{\frac{1}{p}}.
\end{equation}

As mentioned in Section 2, $(\hat{h}^\varepsilon(y_{\lambda_0}^S(t), \theta_t\omega), y_{\lambda_0}^S(t))$ and $(\hat{h}^\varepsilon(y_{\lambda_E}^S(t), \theta_t\omega), y_{\lambda_E}^S(t))$ satisfy the random system (\ref{Equation-f-rds})-(\ref{Equation-s-rds}) with parameter $\lambda_0$, $\lambda_E$ and same initial value $(\hat{h}^\varepsilon(y_0, \omega), y_0)$ respectively, due to the definition of slow manifolds \cite{YD}. Thus, the first term in (\ref{equation-2}) can be handled as follows:
via the variation of constants formula, from
\begin{align}
&\frac{d}{dt}[\hat{h}^\varepsilon(y_{\lambda_0}^S(t), \theta_t\omega)-\hat{h}^\varepsilon(y_{\lambda_E}^S(t), \theta_t\omega)]\notag\\
=&\frac{1}{\varepsilon}\big[f(\hat{h}^\varepsilon(y_{\lambda_0}^S(t), \theta_t\omega)+\sigma\xi(\omega),y_{\lambda_0}^S(t))-f(\hat{h}^\varepsilon(y_{\lambda_E}^S(t), \theta_t\omega)+\sigma\xi(\omega),y_{\lambda_E}^S(t))\big]\notag\\
&+\frac{A}{\varepsilon}[\hat{h}^\varepsilon(y_{\lambda_0}^S(t), \theta_t\omega)-\hat{h}^\varepsilon(y_{\lambda_E}^S(t), \theta_t\omega)]\notag
\end{align}
we obtain
\begin{align}
&\big|\hat{h}^\varepsilon(y_{\lambda_0}^S(t), \theta_t\omega)-\hat{h}^\varepsilon(y_{\lambda_E}^S(t), \theta_t\omega)\big|_{\mathbb{R}^n}\notag\\
=&\Big|\frac{1}{\varepsilon}\int_{0}^{t}e^{\frac{A}{\varepsilon}(t-s)}[f(\hat{h}^\varepsilon(y_{\lambda_0}^S(s), \theta_s\omega)+\sigma\xi(\omega),y_{\lambda_0}^S(s))-f(\hat{h}^\varepsilon(y_{\lambda_E}^S(s), \theta_s\omega)+\sigma\xi(\omega),y_{\lambda_E}^S(s))]ds\Big|_{\mathbb{R}^n}\notag\\
\leq &\frac{KL_f}{\varepsilon}
\int_{0}^{t}e^{\frac{\gamma}{\varepsilon}(s-t)}\big[|y_{\lambda_0}^{S}(s)-y_{\lambda_E}^S(s)|_{\mathbb{R}^m}+|\hat{h}^\varepsilon(y_{\lambda_0}^S(t), \theta_t\omega)-\hat{h}^\varepsilon(y_{\lambda_E}^S(t), \theta_t\omega)|_{\mathbb{R}^n}\big]ds. \notag
\end{align}
Rewrite this as follows
\begin{align}
&e^{\frac{\gamma}{\varepsilon}t}|\hat{h}^\varepsilon(y_{\lambda_0}^S(t), \theta_t\omega)-\hat{h}^\varepsilon(y_{\lambda_E}^S(t), \theta_t\omega)|_{\mathbb{R}^n}\notag\\
\leq& \frac{KL_f}{\varepsilon}
\int_{0}^{t}e^{\frac{\gamma}{\varepsilon}s}|y_{\lambda_0}^{S}(s)-y_{\lambda_E}^S(s)|_{\mathbb{R}^m}ds
+\frac{KL_f}{\varepsilon}\int_{0}^{t}e^{\frac{\gamma}{\varepsilon}s}|\hat{h}^\varepsilon(y_{\lambda_0}^S(t), \theta_t\omega)-\hat{h}^\varepsilon(y_{\lambda_E}^S(t), \theta_t\omega)|_{\mathbb{R}^n}ds. \notag
\end{align}

\noindent By Gronwall's inequality \cite{Duan}, we have
\begin{align}
e^{\frac{\gamma}{\varepsilon}t}|\hat{h}^\varepsilon(y_{\lambda_0}^S(t), \theta_t\omega)-\hat{h}^\varepsilon(y_{\lambda_E}^S(t), \theta_t\omega)|_{\mathbb{R}^n}
\leq\frac{KL_f}{\varepsilon}\int_{0}^{t}e^{\frac{\gamma}{\varepsilon}s}|y_{\lambda_0}^{S}(s)-y_{\lambda_E}^S(s)|_{\mathbb{R}^m}e^{\frac{KL_f}{\varepsilon}(t-s)}ds, \notag
\end{align}
and thus 
\begin{align}
&|\hat{h}^\varepsilon(y_{\lambda_0}^S(t), \theta_t\omega)-\hat{h}^\varepsilon(y_{\lambda_E}^S(t), \theta_t\omega)|_{\mathbb{R}^n}\leq \frac{KL_f}{\varepsilon}\int_{0}^{t}|y_{\lambda_0}^{S}(s)-y_{\lambda_E}^S(s)|_{\mathbb{R}^m}e^{-\frac{\gamma-KL_f}{\varepsilon}(t-s)}ds. \notag
\end{align}
And by exchanging the order of integrals, we obtain that
\begin{align}
\int_{0}^{T}\int_{0}^{t}|y_{\lambda_0}^{S}(s)-y_{\lambda_E}^S(s)|_{\mathbb{R}^m}e^{-\frac{\gamma-KL_f}{\varepsilon}(t-s)}dsdt
&=\int_{0}^{T}\int_{s}^{T}|y_{\lambda_0}^{S}(s)-y_{\lambda_E}^S(s)|_{\mathbb{R}^m}e^{-\frac{\gamma-KL_f}{\varepsilon}(t-s)}dtds\notag\\
&< \frac{\varepsilon}{\gamma-KL_f}\int_{0}^{T}|y_{\lambda_0}^{S}(s)-y_{\lambda_E}^S(s)|_{\mathbb{R}^m}ds. \notag
\end{align}
Thus, by H\"{o}lder inequality, we obtain that
\begin{align}\label{equation-3}
\int_{0}^{T}&|\hat{h}^\varepsilon(y_{\lambda_0}^S(t), \theta_t\omega)-\hat{h}^\varepsilon(y_{\lambda_E}^S(t), \theta_t\omega)|_{\mathbb{R}^n}dt
\leq \frac{KL_f}{\varepsilon}\int_{0}^{T}\int_{s}^{T}|y_{\lambda_0}^{S}(s)-y_{\lambda_E}^S(s)|_{\mathbb{R}^m}e^{-\frac{\gamma-KL_f}{\varepsilon}(t-s)}dtds
\notag\\
&<\frac{KL_f}{\gamma-KL_f}\int_{0}^{T}|y_{\lambda_0}^{S}(s)-y_{\lambda_E}^S(s)|_{\mathbb{R}^m}ds
\leq\frac{KL_f}{\gamma-KL_f}(T^{p-1}\int_{0}^{T}|y_{\lambda_0}^{S}(t)-y_{\lambda_E}^S(t)|_{\mathbb{R}^m}^p)^{\frac{1}{p}}
\end{align}
By inserting (\ref{inequality1}) and (\ref{equation-3}) into (\ref{equation-2}), and taking expectation on two sides, and using  H\"{o}lder inequality $\mathds{E}(|XY|)\leq [\mathds{E}(|X|^p)]^{\frac{1}{p}}[\mathds{E}(|Y|^q)]^{\frac{1}{q}}$ where $p$, $q$ satisfy $\frac{1}{p}+\frac{1}{q}=1$, together with (\ref{pf2}), we have

\begin{align}\label{equation-4}
|\lambda_0-\lambda_E|\cdot H(\lambda_0,\lambda_E)&<(KT^{-\frac{1}{p}}+KL_gT^{\frac{p-1}{p}}+\frac{K^2L_gL_f}{\gamma-KL_f}T^{\frac{p-1}{p}})
\big(\mathds{E}\int_{0}^{T}|y_{\lambda_0}^{S}(t)-y_{\lambda_E}^S(t)|_{\mathbb{R}^m}^pdt\big)^{\frac{1}{p}}\notag\\
\leq&C^{\frac{1}{p}}(KT^{-\frac{1}{p}}+KL_gT^{\frac{p-1}{p}}+\frac{K^2L_gL_f}{\gamma-KL_f}T^{\frac{p-1}{p}})(F(\lambda_0)^{\frac{1}{p}}+F(\lambda_E)^{\frac{1}{p}})
\end{align}
Here, $H(\lambda_0,\lambda_E)=\mathds{E}|\int_{0}^{t^\ast}e^{-Bt}\nabla_\lambda g(\hat{h}^\varepsilon(y_{\lambda'}^S(t), \theta_t\omega)+\sigma\xi(\omega),y_{\lambda'}^S(t),\lambda')dt|_{\mathbb{R}^m}$ where $\lambda'=\lambda_0+\kappa(\lambda_E-\lambda_0)$ with $\kappa\in(0,1).$\par
Note that $\Lambda$ is a closed interval of $\mathbb{R}$ and $H(\lambda_0,\lambda_E)>0$ for $\lambda_E\in \Lambda$, together with the fact that $H(\lambda_0,\lambda_E)$ is continuous with respect to $\lambda_E$ as $\nabla_\lambda g(x,y,\lambda)$ is a continuous function of $x,y$ and $\lambda$, then there exists a $\bar{\lambda}\in \Lambda$ such that $H(\lambda_0,\bar{\lambda})=\max_{\lambda_E\in \Lambda}H(\lambda_0,\lambda_E)$. Hence by inserting (\ref{pf3}) into (\ref{equation-4}), we obtain the desired error estimation for $\lambda_E$
\begin{align}\label{equation-5}
|\lambda_0&-\lambda_E|<\frac{1}{H(\lambda_0,\bar{\lambda})}C^{\frac{1}{p}}(KT^{-\frac{1}{p}}+KL_gT^{\frac{p-1}{p}}+\frac{K^2L_gL_f}{\gamma-KL_f}T^{\frac{p-1}{p}})\notag\\
&\cdot(\frac{C_1\varepsilon}{cp}(\mathds{E}|x_0-\sigma\xi(\omega)-\hat{h}^\varepsilon(y_0,\omega)|_{\mathbb{R}^n}^p)^{\frac{1}{p}}+F(\lambda_E)^{\frac{1}{p}}).\notag
\end{align}
In addition, note that
\begin{align}
&|\tilde{h}^\varepsilon(y_{\lambda_0}^S(t), \theta_t\omega)-\tilde{h}^\varepsilon(y_{\lambda_E}^S(t), \theta_t\omega)|_{\mathbb{R}^n}\notag\\
\leqslant &|\tilde{h}^\varepsilon(y_{\lambda_0}^S(t), \theta_t\omega)-\hat{h}^\varepsilon(y_{\lambda_0}^S(t), \theta_t\omega)|_{\mathbb{R}^n}+
|\hat{h}^\varepsilon(y_{\lambda_0}^S(t), \theta_t\omega)-\hat{h}^\varepsilon(y_{\lambda_E}^S(t), \theta_t\omega)|_{\mathbb{R}^n}\notag\\
&+|\hat{h}^\varepsilon(y_{\lambda_E}^S(t), \theta_t\omega)-\tilde{h}^\varepsilon(y_{\lambda_E}^S(t), \theta_t\omega)|_{\mathbb{R}^n},\notag\\
\end{align}
Thus, we also can use $\tilde{h}^\varepsilon(y, \omega)$ instead of $\hat{h}^\varepsilon(y, \theta_t\omega)$ to get an estimator with error which is also controlled by  $\mathcal{O}((F(\lambda_E))^{\frac{1}{p}})$ and $\mathcal{O}(\varepsilon)$.
The proof is complete.\par
\end{proof}

We remark that using only observations on slow variables facilitates our method. It is often more feasible to observe slow variables than fast variables. In addition, our method reduces the computational complexity and has an important advantage in computational
 cost, since this slow system is lower dimensional than the original system. The results established here offer a benefit of dimension reduction in quantifying parameters in stochastic dynamical systems.

\renewcommand{\theequation}{\thesection.\arabic{equation}}
\setcounter{equation}{0}

\section{Numerical experiments}
\noindent
In this section, we proceed an example in $\mathbb{R}^{2}$ to verify our parameter estimation method based on random slow manifolds \cite{QZD}.\par
Consider the following fast-slow stochastic system
\begin{eqnarray}
\dot{x}&=&-\frac{1}{\varepsilon}x+\frac{1}{4\varepsilon}\cos(y)+\sigma\varepsilon^{-\frac{1}{\alpha}}\dot{L_t^\alpha},\;\;\;\; x(0)=x_0 \in \mathbb{R}, \label{Example-f-main}\\
\dot{y}&=&y+\frac{1}{4}\sin(\lambda x),\;\;\;\; y(0)=y_0 \in \mathbb{R} \label{Example-s-main},
\end{eqnarray}
where $\lambda$ is a real unknown positive parameter, $A=-1$, $B=1$, $f(x,y)=\frac{1}{4}\cos y$, $g(x,y,\lambda)=\frac{1}{4}\sin(\lambda x)$. It is easy to justify that $A$, $B$, $f$, $g$ satisfy ({\bf H1})-({\bf H2}) with $K=\gamma=\beta=1$, $L_f=L_g=\frac{1}{4}$. Thus, the proposed method of this paper is applicable.\par
By the random transformation (\ref{Transformation}) and Lemma \ref{lemma 2.1}, there exists an $\hat{h}^\varepsilon$ satisfying
\begin{equation}
\hat{h}^\varepsilon(\zeta, \omega)=\frac{1}{\varepsilon}\int_{-\infty}^{0}e^{\frac{s}{\varepsilon}}\frac{1}{4}\cos (\hat{y}(s,\omega,\zeta))ds ,\;\;\zeta\in\mathbb{R}.
\end{equation}
\noindent In fact, $\hat{h}^\varepsilon(\zeta, \omega)$ has an approximation $\tilde{h}^\varepsilon(\zeta, \omega)=\hat{h}^{(0)}(\zeta, \omega)+\varepsilon\hat{h}^{(1)}(\zeta, \omega)$ with error $\mathcal{O}(\varepsilon^2)$ by the Lemma \ref{lemma 2.3}. Here, $\hat{h}^{(0)}(\zeta, \omega)=\frac{1}{4}\cos \zeta$ and
$\hat{h}^{(1)}(\zeta, \omega)$ has an explicit expression,
$$\hat{h}^{(1)}(\zeta, \omega)=\frac{1}{4}\zeta\sin\zeta-\frac{1}{16}\sin\zeta\int_{-\infty}^{0}e^t[\int_{0}^{t}\sin(\frac{1}{4}\lambda\cos\zeta+\lambda\sigma\int_{-\infty}^{s}e^{-(s-r)}dL_r^\alpha(\omega))ds]dt.$$\par
\noindent So the approximated slow system is
\begin{equation}\label{exampleslowsystem}
\dot{y}=y+\frac{1}{4}\sin(\lambda[\tilde{h}^\varepsilon(y, \theta_t\omega)+\sigma\xi(\omega)]),
\end{equation}
with $\xi(\omega)=\int_{-\infty}^{0}e^{s}dL_s^\alpha(\omega)$.\par

In the following numerical simulations, we use a stochastic Nelder-Mead method \cite{Chang, RDC, ZCZ} to estimate unknown parameter $\lambda$ in (\ref{exampleslowsystem}). The main idea is as follows: we determine the estimated parameter value by minimizing the objective function $F(\lambda)=\mathds{E}\sum_{j=1}^{K}\sum_{i=0}^{L}|y^{i,j}_{ob}-y^{i}_{S}(\lambda)|^{p}_{\mathbb{R}^{2}}$, where $\{y^{i,j}_{ob}: y^{i,j}_{ob}=y^{i,j}_{ob}(t_i, \lambda_0), i=0,1, ..., L; j=1, 2, ..., K\}$ are $K$ different only observations of the slow component $y$ from the original system with parameter value $\lambda_0$, and observations $\{y^{i}_{S}: y^{i}_{S}=y^{i}_{S}(t_i, \lambda), i=0,1, ..., L; j=1, 2, ..., K\}$  are generated from reduced system corresponding to parameter $\lambda$. These datas are available by using Euler-Maruyama method.

As shown in fig \ref{fig1}, we see that the reduced system on random slow manifold is a good approximation of the slow variable $y$ of the original system. By just about 10 iteration in the stochastic Nelder-Mead method, we  get estimated parameter $\lambda_E\approx 0.9996$, which indicates that our estimator based on random slow manifolds is a good approximation for the true parameter value. 

\begin{figure}
 \centering
 \subfigure[]{

 \includegraphics[width=0.45\textwidth]{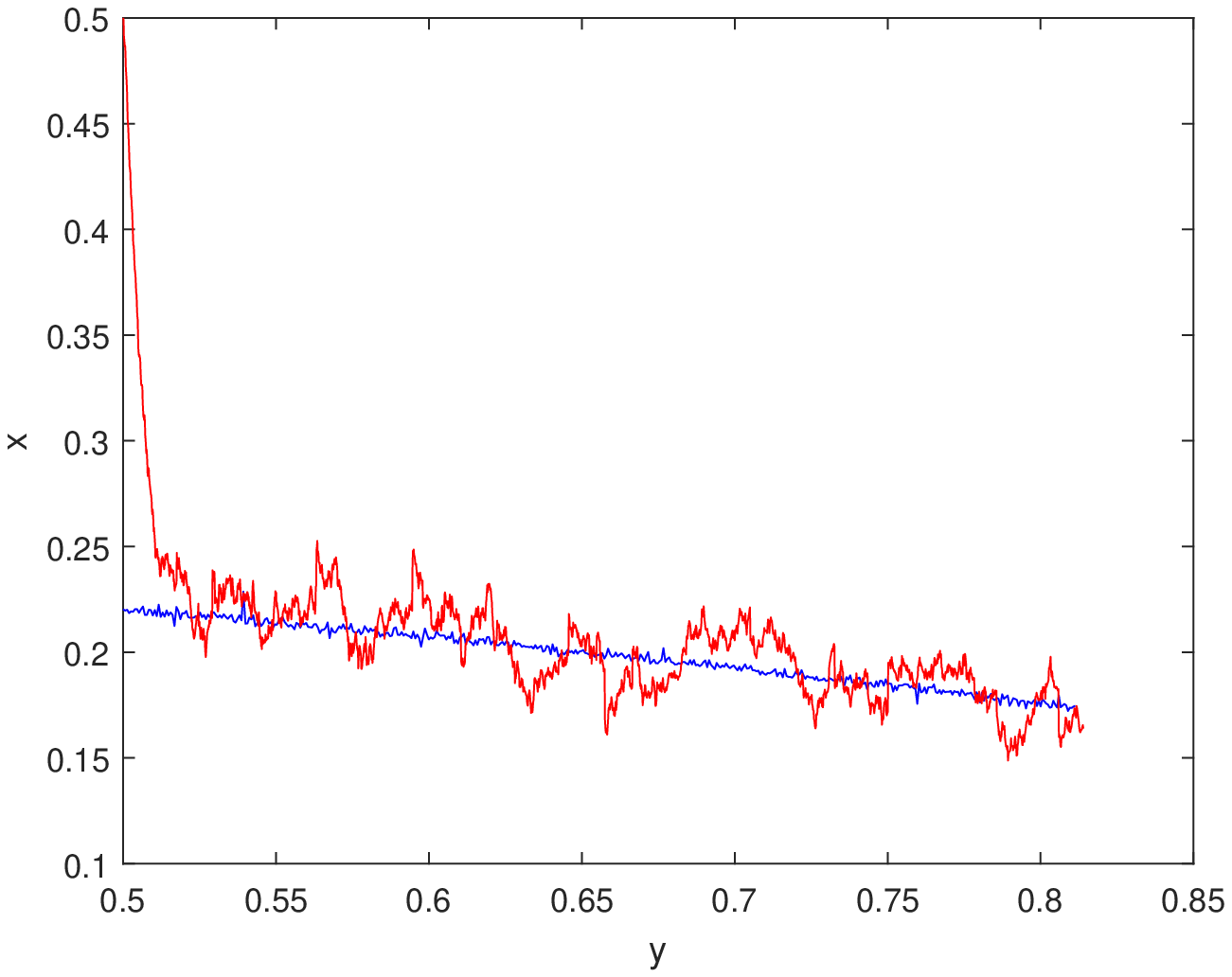}
 }
 \subfigure[]{

 \includegraphics[width=0.45\textwidth]{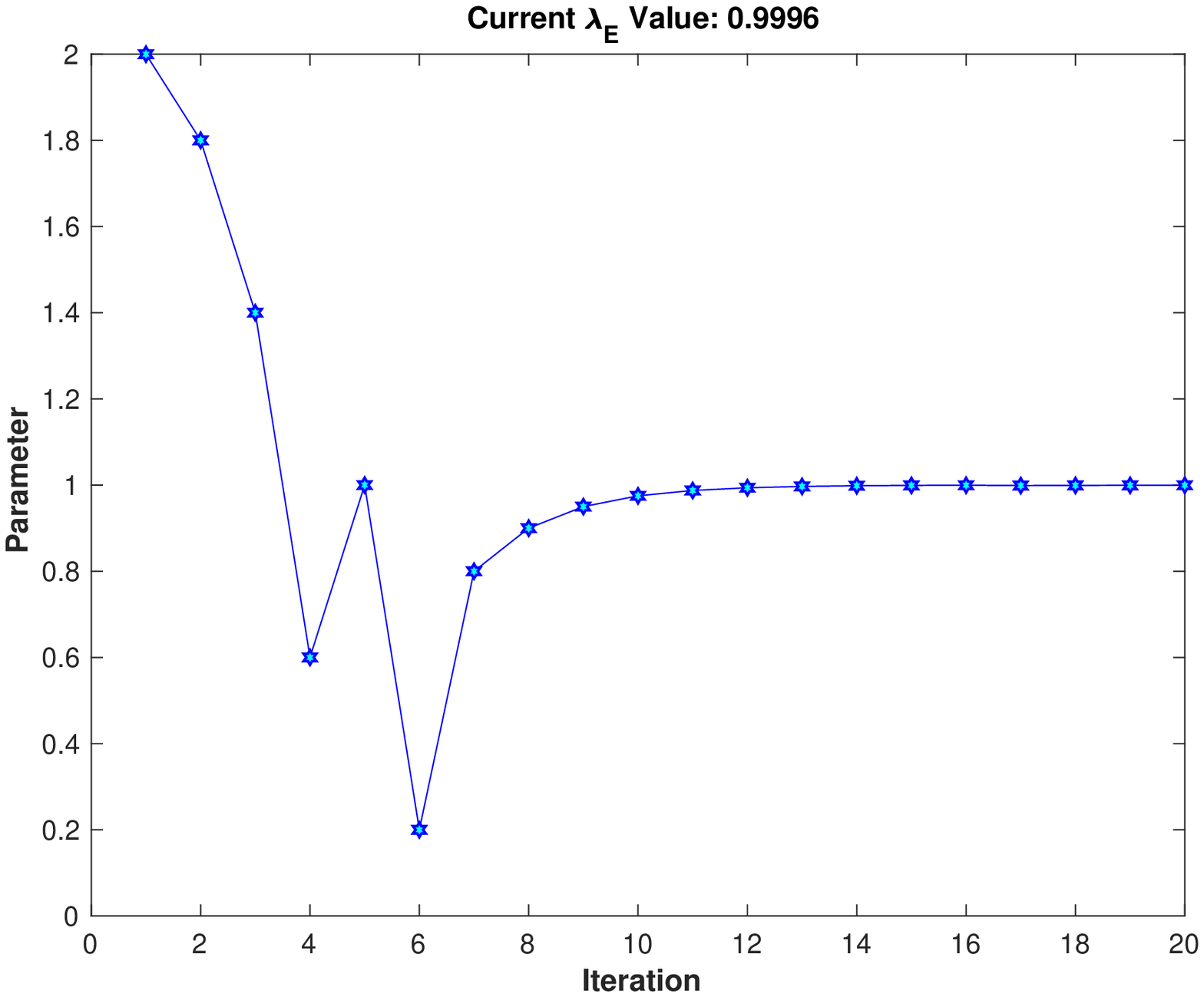}
} 
 \caption{ (Color online) Parameters: $\varepsilon=0.01$, $\alpha=1.8$, $\sigma=0.05$, $\lambda=1$. (a) One orbit in blue of the slow manifold reduced system  tracks the red orbit of system (\ref{Example-f-main})-(\ref{Example-s-main}). (b) The estimated value $\lambda_E$. The simulation result for the parameter $\lambda_E$ is 0.9996.}

 \label{fig1}
 \end{figure}
 
We remark that the parameter method established here can be used to examine complex physical or biological  dynamics, although we illustrate this point by a simple two-dimensional example here. For instance , by the same deduction as here, the unknown parameter in Shimizu-Morioka model under stochastic fluctuations could be determined, see the reference in \cite{Sh}. The parameter in this physical model can capture stochastic bifurcation behaviors \cite{HCDL}. In addition, we are often interested in parameter (see \cite{WTRY}), which represents change rate for mRNA . 
\section{Conclusions and discussion}
We developed a parameter estimation method based on a fast-slow stochastic dynamical system by using the random slow manifolds. Instead of  solving original systems, we can accurately estimate the unknown parameter only by the observation of the slow component, which offers a benefit of computational cost. The results established here can be used to examine biological dynamics, such as stochastic chemical kinetics, where we are more interested in the change rate for mRNA \cite{KJB, TGS, WTRY}.


\begin{acknowledgments}
The authors would like to thank Dr Ziying He, Dr Jianyu Hu and Dr Yanjie Zhang for helpful discussions. This work was partly supported by the NSF grant 1620449,  and NSFC grants 11531006 and 11771449.
\end{acknowledgments}

\bibliography{mybibfile}

\end{document}